\documentclass[10pt]{amsart}
\usepackage{amssymb,amstext,amsmath,amscd,amsthm,amsfonts,enumerate,latexsym,stmaryrd,multicol,geometry,graphicx,mathrsfs,bm}
\usepackage[usenames]{color}
\usepackage[all]{xy}
\newtheorem{thm}{Theorem}[section]
\newtheorem{lem}[thm]{Lemma}
\newtheorem{prop}[thm]{Proposition}
\newtheorem{cor}[thm]{Corollary}
\theoremstyle{definition}
\newtheorem{dfn}[thm]{Definition}

\newtheorem{rem}[thm]{Remark}

\newtheorem{ex}[thm]{Example}

\theoremstyle{remark}

\newtheorem*{claim*}{Claim}
\newtheorem*{ac}{Acknowledgments}
\newtheorem*{conv}{Convention}

\numberwithin{equation}{thm}

\def\add{\operatorname{\mathsf{add}}}

\def\c{\operatorname{\mathsf{C}}}
\def\C{\mathcal{C}}

\def\cone{\operatorname{cone}}
\def\D{\mathrm{D}}

\def\db{\operatorname{\mathsf{D^b}}}
\def\depth{\operatorname{depth}}
\def\dim{\operatorname{dim}}

\def\gt{\operatorname{gt}}

\def\ge{\geqslant}

\def\hom{\operatorname{Hom}}

\def\K{\mathrm{K}}

\def\le{\leqslant}
\def\level{\operatorname{level}}
\def\Lotimes{\otimes^{\mathbf{L}}}
\def\m{\mathfrak{m}}
\def\max{\operatorname{max}}

\def\mod{\operatorname{\mathsf{mod}}}
\def\Mod{\operatorname{\mathsf{Mod}}}
\def\n{\mathfrak{n}}

\def\p{\mathfrak{p}}
\def\perf{\operatorname{\mathsf{perf}}}
\def\proj{\operatorname{\mathsf{proj}}}

\def\soc{\operatorname{soc}}
\def\spec{\operatorname{Spec}}
\def\sup{\operatorname{sup}}

\def\T{\mathcal{T}}

\def\Tor{\operatorname{Tor}}

\def\X{\mathcal{X}}
\def\Y{\mathcal{Y}}

\begin{document}
\title{A lower bound for the Rouquier dimension of derived categories over commutative rings} 
\author{Yuki Mifune}
\address{Graduate School of Mathematics, Nagoya University, Furocho, Chikusaku, Nagoya 464-8602, Japan}
\email{yuki.mifune.c9@math.nagoya-u.ac.jp}
\thanks{The author was partly supported by Grant-in-Aid for JSPS Fellows 25KJ1386.}
\begin{abstract}
We prove that the Rouquier dimension of the bounded derived category of finitely generated modules over a commutative noetherian ring is bounded below by the Krull dimension of the ring.
\end{abstract}
\maketitle
\section{Introduction}
The Rouquier dimension of a triangulated category, introduced by Bondal, Rouquier, and Van den Bergh \cite{BV,Rou}, is an invariant that measures the number of mapping cones needed to build the whole category out of a single object, up to finite direct sums, direct summands, and shifts; see Definition \ref{def_dim} for the precise definition.

Let $R$ be a commutative noetherian ring.
We denote by $\mod R$ the category of finitely generated $R$-modules, and by $\db(R)=\db(\mod R)$ the bounded derived category of $\mod R$.
In this paper, we investigate lower bounds for the Rouquier dimension of $\db(R)$.
In other words, our aim is to estimate the Rouquier dimension of $\db(R)$ in terms of classical invariants from commutative algebra. 
Such lower bounds often play an essential role in determining the exact value of $\dim \db(R)$, especially when they are combined with known upper bounds.
The following theorem summarizes several known results in this direction.
\begin{thm}[Rouquier, Aihara--Takahashi, Dey--Lank--Takahashi, Letz]\label{prev}
Let $R$ be a commutative noetherian ring. Then one has
\[
\dim R \le \dim \db(R), \qquad
\dim R-1 \le \dim \db(R), \qquad
\depth R \le \dim \db(R)
\]
when $R$ is, respectively, a reduced finitely generated algebra over a field, a reduced ring, and a local ring.
\end{thm}
Let us explain this theorem.
Rouquier first showed that if $R$ is a reduced finitely generated algebra over a field, then $\dim R$ gives a lower bound for $\dim\db(R)$; see \cite[Proposition 7.16]{Rou}.
In particular, for regular finitely generated algebras over a field, this lower bound together with the standard upper bound coming from finite global dimension gives the equality $\dim R=\dim\db(R)$.
Aihara--Takahashi later refined Rouquier's result and extended this type of lower bound beyond the finitely generated case. 
In particular, they proved that if $R$ is a reduced noetherian ring of finite Krull dimension, then $\dim R-1$ gives a lower bound for $\dim\db(R)$; see \cite[Theorem 6.2 and Corollary 6.6]{AT2015}. Furthermore, Dey--Lank--Takahashi removed the assumption that $R$ has finite Krull dimension from this statement; see \cite[Lemma 2.5]{DLT}.
More recently, Kekkou--Letz--Stephan \cite{KLS} systematically developed the theory of regular sequences in $R$-linear triangulated categories. 
In this context, Letz \cite{L} proved that the length of a regular sequence gives a lower bound for the generation time in an $R$-linear triangulated category. 
Note that the third assertion in Theorem \ref{prev} follows by applying \cite[Proposition 2.5]{L} to the opposite category of $\db(R)$ with $M=R$. 
In particular, $\dim R$ gives a lower bound for $\dim\db(R)$ in the Cohen--Macaulay case.

The main result of this paper is the following:
\begin{thm}[Corollary \ref{dim}]\label{mthm_int}
For an arbitrary commutative noetherian ring $R$, one has
\[
\dim R \le \dim\db(R).
\]
\end{thm}
This gives a common generalization of the results collected in Theorem \ref{prev}.
Moreover, it shows that the finiteness of $\dim\db(R)$ forces the finiteness of $\dim R$.
The proof is based on the method used in Letz's work on lower bounds arising from regular sequences \cite{L}, together with the existence of balanced big Cohen--Macaulay algebras \cite{Andre2018}.

The organization of this paper is as follows. 
In Section 2, we give a proof of Theorem \ref{mthm_int}. 
In Section 3, as an application of the main theorem, we study rings for which the lower bound on $\dim\db(R)$ given by $\dim R$ is optimal.
\begin{conv}
Throughout this paper, $R$ denotes a commutative noetherian ring.
All subcategories are assumed to be strictly full.
\end{conv}
\section{Proof of the main theorem}
In this section, we recall definitions related to the Rouquier dimension of a triangulated category and provide a proof of the main result.
We begin by recalling the definitions of level, generation time, and Rouquier dimension in a triangulated category.
\begin{dfn}\label{def_dim}
Let $\C$ be an additive category and $\T$ a triangulated category.
\begin{enumerate}[\rm(1)]
\item
For a subcategory $\X$ of $\C$, we denote by $\add_{\C}\X$ the {\em additive closure} of $\X$ in $\mathcal{C}$, that is, the subcategory of $\mathcal{C}$ consisting of direct summands of finite direct sums of objects in $\X$.
\item
For subcategories $\X,\Y$ of $\T$, we denote by $\X\ast\Y$ the subcategory of $\T$ consisting of objects $E$ such that there exists an exact triangle $X \to E \to Y \to$ in $\T$ with $X\in \X$ and $Y \in \Y$.
\item
Let $\X$ be a subcategory of $\T$.
We set ${\langle\X\rangle}_{0}^{\T}=0$, and ${\langle\X\rangle}_{1}^{\T}=\add_{\T}\{X[n]\mid n \in \mathbb{Z} \text{ and } X\in\X\}$.
For an integer $r>1$, we inductively define ${\langle\X\rangle}_{r}^{\T}$ to be $\langle{{\langle\X\rangle}_{r-1}^{\T}\ast {\langle\X\rangle}_{1}^{\T}\rangle}_{1}^{\T}$.
\item
For a subcategory $\X$ of $\T$ and an object $M$ in $\T$, we define the $\X$-{\em level} of $M$ in $\T$, denoted by $\level_{\T}^{\X}M$, as the infimum of integers $n\ge0$ such that $M\in{\langle\X\rangle}_{n}^{\T}$.
\item
For an object $G\in\T$, the {\em generation time} of $G$ is defined as $\gt(G)=\sup\{\level_{\T}^{G}X-1\mid X\in\T\}$.
We say that $G$ is a {\em strong generator} of $\T$ if $\gt(G)<\infty$.
\item
The {\em Rouquier dimension} of $\T$, denoted by $\dim\T$, is defined as the infimum of $\gt(G)$, where $G$ is a strong generator of $\T$.
\end{enumerate}
\end{dfn}
The following elementary lemma will be useful in constructing nonzero morphisms in the derived category. 
Its proof is straightforward.
\begin{lem}\label{lem_alg}
Let $B$ be an $R$-algebra and $x_1,\ldots,x_d$ elements of $R$ such that the sequence $x_1,\ldots,x_d$ forms a $B$-sequence.
Then for all positive integers $n_1,\ldots,n_d$, one has $x_{1}^{n_{1}}\cdots x_{d}^{n_{d}} \cdot 1_{B}\notin (x_{1}^{n_{1}+1},\ldots,x_{d}^{n_{d}+1})B$.
\end{lem}
Let $\T$ be an $R$-linear triangulated category and $G$ an object in $\T$.
We set $\hom_{\T}^{*}(-,G)=\bigoplus_{i\in\mathbb{Z}}\hom_{\T}(-,G[i])$.
Note that the functor $\hom_{\T}^{*}(-,G):\T^{\mathrm{op}}\to \Mod R$ is a cohomological functor.
Recall that a morphism $f:M\to N$ in $\T$ is {\em $G$-coghost} if the map $\hom_{\T}^{*}(f,G)$ is zero.
The following lemma follows from the argument dual to that of \cite[Lemma 2.3]{L}.
\begin{lem}\label{lem_coghost}
Let $\T$ be an $R$-linear triangulated category and $G,N$ objects in $\T$.
Let $x$ be an element in $R$ such that $x^n\Gamma_{x}H=0$ for some $n>0$, where $H=\hom_{\T}^{*}(N,G)$.
Then the composition $\delta(x^{n+1})\circ (x^{n}\cdot 1_{N})$ is $G$-coghost, where $\delta(x^{n+1}):N\to \cone(x^{n+1}\cdot 1_{N})$ is the canonical morphism.
\end{lem}
We have reached the main result of this section.
The proof proceeds by using the lemmas stated above and the existence of a balanced big Cohen--Macaulay algebra to construct $\dim R$ nonzero coghost maps; the assertion then follows from the coghost lemma.
\begin{thm}\label{mthm}
Let $(R,\m)$ be a $d$-dimensional noetherian local ring and  $x_{1},\ldots,x_{d}$ a system of parameters of $R$.
Then for any object $G\in\db(R)$, there exist positive integers $n_1,\ldots,n_d$ such that $\level_{\db(R)}^{G}\K(x_{1}^{n_{1}+1},\ldots,x_{d}^{n_{d}+1})\ge d+1$.
In particular, one has $d\le \gt(G)$, and we obtain that $d\le\dim\db(R)$.
\end{thm}
\begin{proof}
Set $E_0=R$. 
We shall construct objects $E_i\in \perf(R)$ and morphisms $\delta_{i}:E_{i-1}\to E_{i}$ for $1\le i\le d$ inductively. 
Suppose that $E_{i-1}$ has been constructed. 
Since $E_{i-1}\in \perf(R)$ and $G\in \db(R)$, the $R$-module $\hom_{\db(R)}^{*}(E_{i-1},G)$ is finitely generated. 
Hence there exists an integer $n_{i}>0$ such that $x_{i}^{n_i}\Gamma_{x_{i}}\hom_{\db(R)}^{*}(E_{i-1},G)=0$. 
We set $E_{i}=\cone(x_{i}^{n_{i}+1}\cdot 1_{E_{i-1}})$ and let $\delta_{i}:E_{i-1}\to E_{i}$ be the canonical morphism.
By Lemma \ref{lem_coghost}, the morphism $\delta_{i}\circ (x_{i}^{n_i}\cdot 1_{E_{i-1}}):E_{i-1}\to E_{i}$ is $G$-coghost. 
By construction, $E_{d}$ is the Koszul complex $\K(x_{1}^{n_{1}+1},\ldots,x_{d}^{n_{d}+1})$. 
Let $e=\delta_{d}\circ\cdots\circ\delta_{1}:R\to E_{d}$, and set $\alpha=(\delta_{d}\circ (x_{d}^{n_{d}}\cdot 1_{E_{d-1}}))\circ\cdots\circ(\delta_{1}\circ (x_{1}^{n_{1}}\cdot 1_{E_{0}}))=x_{1}^{n_{1}}\cdots x_{d}^{n_{d}}\cdot e$. 
Then $\alpha$ is a composition of $d$ $G$-coghost maps.
It remains to show that $\alpha\ne 0$ in $\db(R)$. 
Let $B$ be a balanced big Cohen--Macaulay $R$-algebra, whose existence follows from \cite{Andre2018}.
Applying $-\Lotimes_R B$ to $\alpha$, we obtain a morphism $\alpha\Lotimes_R B: B \to E_{d}\Lotimes_R B$ in $\D(\Mod B)$. 
Under the standard identification of $E_{d}\Lotimes_R B$ with the Koszul complex on $x_{1}^{n_{1}+1},\ldots,x_{d}^{n_{d}+1}$ over $B$, the map $H^0(\alpha\Lotimes_R B):B\to B/(x_{1}^{n_{1}+1},\ldots,x_{d}^{n_{d}+1})B$ sends $1_B$ to $\overline{x_{1}^{n_{1}}\cdots x_{d}^{n_{d}}\cdot 1_{B}}$.
Since $B$ is balanced big Cohen--Macaulay, the sequence $x_{1},\ldots,x_{d}$ is $B$-regular, and hence the element $\overline{x_{1}^{n_{1}}\cdots x_{d}^{n_{d}}\cdot 1_{B}}$ is nonzero by Lemma \ref{lem_alg}.
Thus $\alpha\Lotimes_R B\ne 0$, and therefore $\alpha\ne 0$.
Now the coghost lemma \cite[Lemma 4.11]{Rou} implies that $\level_{\db(R)}^{G} E_{d}\ge d+1$.
This proves the desired inequality. 
The inequalities $d\le \gt(G)$ and $d\le \dim\db(R)$ follow immediately from the definitions of generation time and Rouquier dimension.
\end{proof}
The following result is a direct consequence of Theorem \ref{mthm}.
\begin{cor}\label{dim}
Let $R$ be a commutative noetherian ring.
Then one has
\[
\dim R\le \dim\db(R).
\]
In particular, if $\dim\db(R)<\infty$, then $\dim R<\infty$.
\end{cor}
\begin{proof}
For any prime ideal $\p$ of $R$, we have $\dim R_{\p} \le \dim\db(R_{\p}) \le \dim\db(R)$, where the first inequality follows from Theorem \ref{mthm} and the second one follows from \cite[Lemma 4.2]{AT2015}.
Taking the supremum over all prime ideals $\p$, we obtain the desired conclusion.
\end{proof}
\begin{rem}
The lower bound on $\dim\db(R)$ in terms of $\dim R$ given in Corollary \ref{dim} is not optimal in general; see \cite[Remark 7.18]{Rou} and \cite[Corollary 5.10]{BIKO2010} for instance.
\end{rem}
\section{Examples}
In this section, we give a criterion for a local ring, not necessarily Cohen--Macaulay, to have finite $d$-syzygy representation type. 
This criterion yields examples for which the lower bound $\dim R\le \dim\db(R)$ is optimal.
We first consider the class of local rings of finite $d$-syzygy representation type, which is a generalization of finite Cohen--Macaulay representation type.
We refer the reader to \cite{DKLO,M2} for details on finite syzygy representation type.
\begin{dfn}
Let $R$ be a commutative noetherian ring.
\begin{enumerate}[\rm(1)]
\item
For a subcategory $\X$ of $\mod R$ and $n>0$, we denote by $\Omega^{n}\X$ the subcategory of $\mod R$ consisting of modules $M$ such that there exists an exact sequence $0 \to M \to F_{n-1} \to \cdots \to F_0 \to X \to 0$ with each $F_i \in \proj R$ and $X \in \X$.
We set $\Omega^0 \X=\X$.
When $R$ is a local ring, for $M \in \mod R$ we denote by $\Omega^n M$ the $n$-th syzygy of $M$ in its minimal free resolution.
\item
For an integer $n\ge 0$, we say that $R$ is {\em of finite $n$-syzygy representation type} if $\Omega^n(\mod R)\subseteq \add G$ for some $G\in\mod R$.
\end{enumerate}
\end{dfn}
\begin{rem}
Let $R$ be a local ring and $n$ a nonnegative integer.
\begin{enumerate}[\rm(1)]
\item
If $n>0$, then one has $\Omega^n(\mod R)=\{R^{\oplus m} \oplus \Omega^{n}M\mid m\ge 0, M\in\mod R\}$.
\item
By \cite[Theorem 2.2]{LW}, $R$ is of finite $n$-syzygy representation type if and only if $\add\Omega^n(\mod R) = \add G$ for some $G\in\mod R$.
\item
Suppose that $R$ is of finite $n$-syzygy representation type for some $n\ge 0$.
Then one has $\dim R\le n$ by \cite[Remark 4.7]{DLMO}.
Moreover, $R$ has an isolated singularity; see \cite[Theorem 3.7]{DKLO}, \cite[Theorem 3.8]{M2}, and \cite[Theorem 2.2]{LW}.
\item
If $R$ is a $d$-dimensional Cohen--Macaulay local ring, then $R$ is of finite $d$-syzygy representation type if and only if $R$ is of finite Cohen--Macaulay representation type; see \cite[Proposition 12.8]{LW}.
\end{enumerate}
\end{rem}
The following result shows that, for a $d$-dimensional noetherian ring $R$ of finite $d$-syzygy representation type, one can give an explicit upper bound for $\dim\db(R)$ in terms of $d$.
Moreover, $\dim\db(R)=d$ when $R$ is a Cohen--Macaulay local ring of finite Cohen--Macaulay representation type of positive dimension.
\begin{prop}\label{d-syz}
Let $R$ be a $d$-dimensional noetherian ring of finite $d$-syzygy representation type.
Then one has $d\le \dim\db(R)\le \max\{1,2(d-1)\}$.
In particular, if $d\in\{1,2\}$, then we have $\dim\db(R)=d$.
If, moreover, $R$ is a Cohen--Macaulay local ring, then one has $d\le \dim\db(R)\le \max\{1,d\}$.
\end{prop}
\begin{proof}
By Corollary \ref{dim}, we have $\dim R \le \dim \db(R)$. 
On the other hand, by \cite[Proposition 2.6]{AAITY}, we have $\db(R)={\langle \Omega^{2}\mod R\rangle}_{1}^{\db(R)} \ast {\langle \Omega\mod R\rangle}_{1}^{\db(R)}$.
Assume that $\Omega^d(\mod R)\subseteq \add G$ for some $G\in\mod R$. 
If $d\le 1$, then $\db(R)={\langle\Omega\mod R\rangle}_{2}^{\db(R)}={\langle G\rangle}_{2}^{\db(R)}$, and hence $\dim\db(R)\le 1$.
Suppose now that $d\ge 2$. 
Then $\Omega^2(\mod R)\subseteq {\langle G\oplus R\rangle}_{d-1}^{\db(R)}$ and $\Omega(\mod R)\subseteq {\langle G\oplus R\rangle}_{d}^{\db(R)}$. 
It follows that $\db(R)={\langle G\oplus R\rangle}_{2d-1}^{\db(R)}$, and therefore $\dim\db(R)\le 2(d-1)$.
If, moreover, $R$ is a Cohen--Macaulay local ring, then $R$ is of finite Cohen--Macaulay representation type, and one has $\db(R)={\langle M\rangle}_{\max\{2,d+1\}}^{\db(R)}$ for some maximal Cohen--Macaulay $R$-module $M$ by \cite[Theorem 4.1]{AAITY}.
\end{proof}
\begin{rem}
\begin{enumerate}[\rm(1)]
\item
When $R$ is a Cohen--Macaulay local ring, the lower bound in Proposition \ref{d-syz} also follows from the third assertion of Theorem \ref{prev}, which is due to Letz, without using Corollary \ref{dim}.
\item
Combining \cite[Theorem 4.1]{AAITY} with \cite[Proposition 3.2(1)]{DT2023}, one has $\db(R)={\langle \c(R)\rangle}_{\max\{2,d+1\}}^{\db(R)}$, where $\c(R)$ denotes the subcategory of $\mod R$ consisting of modules $M$ such that $\depth M_{\p}\ge\depth R_{\p}$ for all $\p\in\spec R$.
However, if $R$ is not Cohen--Macaulay, then $\c(R)$ is not of finite type; see \cite[Remark 4.7]{DLMO}.
\end{enumerate}
\end{rem}
In the rest of this section, we construct rings of finite $d$-syzygy representation type. 
We first establish the following lemma.
For subcategories $\X,\Y$ of $\mod R$, we denote by $\X \ast \Y$ the subcategory of $\mod R$ consisting of modules $E$ such that there exists an exact sequence $0 \to X \to E \to Y \to 0 $ with $X\in\X$ and $Y\in\Y$.
\begin{lem}\label{lem_syz}
Let $(R,\m,k)$ be a noetherian local ring and $M$ a finitely generated $R$-module.
Set $I=\soc R$ and $S=R/I$.
Then the following hold:
\begin{enumerate}[\rm(1)]
\item
One has $\Omega_{R}M\in (\add k)\ast \Omega_{S}(\mod S)$ in $\mod S$.
\item
One has $\Omega_{S}^{n}\Omega_{R}M\in\add_{S}\{S,(\add\Omega_{S}^{n}k)\ast(\Omega_{S}^{n+1}\mod S)\}$ for all $n\ge 0$.
\item
If $S$ is a regular local ring with $d=\dim S\ge 1$, then we have $\Omega_{S}^{d-1}\Omega_{R}M\in\add_{S}(S\oplus \Omega_{S}^{d-1}k)$.
\end{enumerate}
\end{lem}
\begin{proof}
(1) Since $\Omega_{R}M$ is a submodule of $\m^{\oplus \mu(M)}$, there exists an exact sequence $0\to \Omega_{R}M\cap I^{\oplus \mu(M)}\to \Omega_{R}M \to (\Omega_{R}M+I^{\oplus \mu(M)})/I^{\oplus \mu(M)}\to 0$.
Since $\Omega_{R}M\cap I^{\oplus \mu(M)}$ is annihilated by $\m$ and $(\Omega_{R}M+I^{\oplus \mu(M)})/I^{\oplus \mu(M)}$ is a submodule of a free $S$-module, the assertion follows.
(2) For $n\ge 0$, by taking $n$-th syzygies over $S$ in the exact sequence in (1), we obtain the conclusion.
(3) Since $\Omega_{S}^{d}(\mod S)=\add S$, every short exact sequence ending in an object of $\Omega_{S}^{d}(\mod S)$ in $\mod S$ splits.
Hence, by taking $d-1$ as $n$ in (2), the assertion follows.
\end{proof}
The following proposition gives a sufficient condition for a local ring to have finite $d$-syzygy representation type.
\begin{prop}\label{criterion}
Let $(R,\m,k)$ be a noetherian local ring with $d=\dim R>0$ and $I$ a nonzero ideal of $R$ contained in $\soc R$.
Assume that $t=\depth R/I>0$.
\begin{enumerate}[\rm(1)]
\item
One has $I=\soc R=\Gamma_{\m}R$.
\item
Let $x_{1},\ldots,x_{t}$ be elements in $\m$ such that $\overline{x_{1}},\ldots,\overline{x_{t}}$ forms an $R/I$-sequence.
Then we have $(x_{1},\ldots,x_{t})\cap I=(x_{1},\ldots,x_{t}) I=0$.
\item
Assume that $R/I$ is a regular local ring of dimension $d$, and $\overline{x_{1}},\ldots,\overline{x_{d}}$ is a regular system of parameters of $R/I$.
Then $\m=(x_{1},\ldots,x_{d})\oplus I$ and $R$ is of finite $d$-syzygy representation type.
\end{enumerate}
\end{prop}
\begin{proof}
(1) Since $I\subseteq \soc R\subseteq \Gamma_{\m}R$ and $\Gamma_{\m}(R/I)=0$, one has $I=\Gamma_{\m}R$.
Hence, the above inclusions are equalities.
(2) Since $x_{1},\ldots,x_{t}$ is an $R/I$-sequence, we have $0=\Tor_{1}^{R}(R/I,R/(x_{1},\ldots,x_{t}))\cong I\cap (x_{1},\ldots,x_{t})/I(x_{1},\ldots,x_{t}) = I\cap (x_{1},\ldots,x_{t})$.
(3) Since $\overline{x_1},\ldots,\overline{x_d}$ is a regular system of parameters of $R/I$, we have $\m=(x_1,\ldots,x_d)+I$. 
By (2), $(x_1,\ldots,x_d)\cap I=0$, and hence $\m=(x_1,\ldots,x_d)\oplus I$.
Let $M$ be a finitely generated $R$-module.
Set $S=R/I$.
By \cite[Lemma 3.2]{NT2020}, we have $\Omega_{R}^{d-1}(\Omega_{R}M)\cong\Omega_{S}^{d-1}\Omega_{R}M \oplus \bigoplus_{i=0}^{d-2}\Omega_{S}^{d-2-i}I^{\oplus}$.
It follows from Lemma \ref{lem_syz}(3) that $\Omega_{S}^{d-1}\Omega_{R}M\in\add_{S}(S\oplus \Omega_{S}^{d-1}k)$.
Hence, we have $\Omega_{R}^{d}M \in \add_{S}(S\oplus \Omega_{S}^{d-1}k\oplus \bigoplus_{i=0}^{d-2}\Omega_{S}^{d-2-i}I)$.
Thus, one has $\Omega^{d}(\mod R)\subseteq \add_{R}(R\oplus G)$, where $G=S\oplus \Omega_{S}^{d-1}k\oplus \bigoplus_{i=0}^{d-2}\Omega_{S}^{d-2-i}I$.
\end{proof}
\begin{rem}
\begin{enumerate}[\rm(1)]
\item
In the setting of Proposition \ref{criterion}, the ring $R$ is not Cohen--Macaulay.
\item
In the situation of Proposition \ref{criterion}(3), for each $X\in\db(R)$, the exact triangle $IX\to X\to X/IX\to\quad$ in $\db(R)$ shows that $\db(R)={\langle k\oplus R/I\rangle}_{d+2}^{\db(R)}$.
Hence, together with Corollary \ref{dim}, we obtain $d\le \dim\db(R)\le d+1$.
\end{enumerate}
\end{rem}
We end this section by applying Proposition \ref{criterion} to give examples of rings of finite $d$-syzygy representation type.
\begin{ex}\label{triv}
Let $(S,\n,k)$ be a regular local ring with $d=\dim S>0$ and $n$ a positive integer.
Consider the idealization $R=S\ltimes k^{\oplus n} \cong S[\![x_{1},\ldots,x_{n}]\!]/(\n x_{i},x_{i}x_{j}\mid 1\le i,j\le n)$.
Set $I=(x_{1},\ldots,x_{n})R$.
Then we have $0\neq I\subseteq \soc R$ and $R/I\cong S$.
Hence, by Proposition \ref{criterion}, $R$ is of finite $d$-syzygy representation type.
Thus, if $d\in\{1,2\}$, then it follows from Proposition \ref{d-syz} that $\dim\db(R)=d$.
\end{ex}
\begin{ac}
The author would like to thank his supervisor Ryo Takahashi for giving many thoughtful questions and helpful discussions.
\end{ac}

\end{document}